\documentclass[a4paper,11pt]{amsart}

\usepackage{amsmath,amssymb,amsthm}
\usepackage{mathtools}
\usepackage{amsfonts}

\usepackage{graphicx}
\usepackage{subcaption}   

\usepackage{tikz}
\usetikzlibrary{intersections,calc,arrows.meta}

\usepackage[
  top=30mm,
  bottom=30mm,
  left=25mm,
  right=25mm
]{geometry}

\theoremstyle{definition}

\theoremstyle{definition}
\newtheorem{dfn}{Definition}[section]

\newtheorem{rem}[dfn]{Remark}

\theoremstyle{plain}
\newtheorem{lem}[dfn]{Lemma}

\newtheorem{theo}[dfn]{Theorem}
\newtheorem{cor}[dfn]{Corollary}

\newtheorem*{theono}{Theorem} 

\newcommand{\Var}{\text{Var}}

\newcommand{\Nt}{\{N_t\}_{t \ge 0}}
\newcommand{\Dt}{\{D_t\}_{t \ge 0}}

\allowdisplaybreaks[4]

\title[]{ Laplace Transforms of Stopping Times 
for Subordinator with Applications to Inventory Control 
}

\author{
 Ryoya Koide 
 }

 \address{
Department of Mathematics, Faculty of Science and Technology, Tokyo University of Science, 2641 Yamazaki, Noda-shi,
Chiba prefecture, 278-8510, Japan.
}
\email{
koideryoya@gmail.com
}

\begin{document}
\maketitle
\begin{abstract}
\noindent Intermittent demand fluctuations pose significant challenges in disaster logistics and medical supply systems. In this study, we formulate cumulative demand as a generalized Lévy process composed of a drift term, Poisson jumps, and compound Poisson jumps, and analyze a continuous-time inventory model. The proposed framework provides a unified formulation that encompasses both drifted Poisson processes and drifted compound Poisson processes.
\noindent From a mathematical perspective, we treat the reorder time as a first-passage problem of a subordinator and derive its Laplace transform via the Laplace exponent. In particular, for the drifted Poisson case, we obtain an explicit representation of the inverse Laplace exponent using the Lambert W function, which yields an analytic expression for the Laplace transform of the first-passage time. Furthermore, when the jump sizes follow exponential and Gamma distributions, we derive explicit formulas for the mean and variance of the reorder times, thereby clarifying the moment structure of first-passage times for generalized Lévy demand processes.
\noindent From an operations research perspective, we explicitly characterize the expected total cost over a finite time horizon based on the distribution of cumulative demand. This study presents an analytical framework that integrates first-passage theory of Lévy processes with continuous-time inventory control.
\end{abstract}
\section{Introduction}

\noindent Inventory control problems under uncertain demand constitute an important research area in both probability theory and operations research. In particular, in situations such as disaster response or intermittent demand, continuous demand models are often insufficient. Lévy processes provide a well-behaved class of stochastic processes capable of modeling discontinuous changes. (~\cite{WANG20241038},~\cite{KOURENTZES2013198},~\cite{yuna2023inventory},~\cite{koide2025inventorycontrolusinglevy},~\cite{Koide2025DriftedPoisson}) 

\noindent For example, Noba and Yamazaki (2025)~\cite{noba2025stochastic} propose an inventory control policy based on Lévy processes and conduct analytical investigations. In their framework, the total cost incurred in inventory control is expressed as a general functional form, and the optimal replenishment strategy is characterized. However, difficulties remain in performing numerical computations and explicit calculations under concrete specifications.

\noindent On the other hand, Koide et al. (2024)~\cite{Koide2024} and Koide et al. (2025)~\cite{koide2025inventorycontrolusinglevy} propose inventory control models driven by Lévy processes. By imposing engineering-oriented assumptions that facilitate computation and simulation, they explicitly derive the expected total cost and enable numerical simulations.

\noindent In this study, we consider $\{D_t\}_{t \ge 0}$ as a stochastic process representing cumulative demand. In particular, we focus on non-decreasing Lévy processes (subordinators), such as a drifted Poisson process (cf. the setting in Koide~\cite{koide2025inventorycontrolusinglevy}). Let $a>0$, $Q>0$, and $n \in \mathbb{N}$. The first passage time of $D_t$ is defined by
\begin{equation*}
T_n =  \inf \left\{s > 0 \,\middle|\, D_s \ge a + (n-1)Q \right\}.
\end{equation*}
The random variable $T_n$ represents the time of the $n$-th replenishment. Let $x>0$ denote the initial inventory level, where $x-a$ is the reorder point and $Q>0$ is the order quantity. We consider a fixed-order-quantity policy. Then, the inventory level at time $t$, denoted by $X_t$, is given by
\begin{equation*}
X_t = x - D_t + Q R_t  \quad  \text{, where} \quad 
R_t =  \sum_{n \ge 1}  1_{\{T_n < t \}} 
= \inf\{ k \mid T_k > t\}.
\end{equation*}
Here, $\{R_t\}_{t \ge 0}$ represents the stochastic process describing the cumulative number of orders placed up to time $t$.

\noindent In order to compute the expectation and variance of $X_t$, as well as the expected total cost, it is necessary to evaluate the expectation and variance of $T_n$. However, the distribution of $T_n$ is generally unknown, making direct computation difficult.

\noindent The computational tractability achieved in Koide et al. (2025)~\cite{koide2025inventorycontrolusinglevy} relies on engineering-oriented assumptions motivated by numerical computation and simulation. Moreover, the applicability to real data is discussed in Koide et al. (2025)~\cite{Koide2025DriftedPoisson}.

\noindent In contrast, the present study does not impose such engineering assumptions. Instead, by directly exploiting the probabilistic structure of stopping times for Lévy processes, we demonstrate that the expectation and variance of $T_n$ can be derived explicitly. Furthermore, we analyze the expectation of the inventory level, the expected total cost, and the asymptotic behavior of the total cost.

\section{Setting}

\subsection{Fixed-Order-Quantity Policy under Intermittent Demand}
\quad \\
\noindent Intermittent demand refers to a demand pattern in which demand does not occur continuously, but instead exhibits sudden arrivals after long periods of no demand. 
Such demand is characterized by strong uncertainty in both the timing of demand occurrences and the demand sizes. 
We model the cumulative demand under intermittent demand by a subordinator $\{D_t\}_{t \ge 0}$.
\noindent In this model, the inventory level at time $t$ is given by the initial inventory minus the cumulative demand up to time $t$. 
Whenever the inventory level falls below a prescribed threshold, a replenishment of fixed size $Q$ is placed. 
Since the time at which the inventory reaches the threshold is stochastic, we define the corresponding stopping times.
\noindent To model replenishment at each reorder point, the system is designed so that an order is placed whenever the cumulative demand exceeds a prescribed level. The order quantity is also incorporated explicitly. 
The formulation follows \cite{koide2025inventorycontrolusinglevy}.
\begin{dfn}[\textbf{First Passage Time of $\{D_t\}$}]
Let $a>0$, $Q>0$, and $n \in \mathbb{N}$. 
The first passage time of the cumulative demand process $\{D_t\}$ is defined by
\begin{equation}
T_n^{a,Q} (= T_n )
= \inf \left\{s > 0 \,\middle|\, D_s \ge a + (n-1)Q \right\}.
\end{equation}
\end{dfn}
\noindent The parameter $a$ determines the reorder level, and $Q$ represents the order quantity. 
The random variable $T_n$ denotes the stochastic time of the $n$-th replenishment under the fixed-order-quantity policy. 
Using these quantities, we define the inventory control model under intermittent demand as follows.
\begin{dfn}[\textbf{Inventory Level Process}]
Let $x>0$ denote the initial inventory level. 
We consider a fixed-order-quantity policy with reorder point $x-a$ and order quantity $Q>0$. 
Let $\{D_t\}$ be the subordinator representing cumulative demand. 
The inventory level process $\{X_t\}$ is defined by
\begin{equation}
X_t = x - D_t + Q R_t ,
\end{equation}
where
\begin{equation}
R_t =  \sum_{n \ge 1}  1_{\{T_n^{a,Q} <t \}} 
= \inf\{ k \mid T_k^{a,Q} > t\}
\end{equation}
denotes the stochastic process representing the cumulative number of orders placed up to time $t$.
\end{dfn}
\noindent Taking expectations yields
\begin{align}
E[X_t] 
&= x - E[D_t] + Q E[R_t] \notag\\
&= x - E[D_t] + Q \sum_{n \ge 1} E[1_{\{T_n< t\}}] 
\quad \text{(by Fubini's theorem)} \notag\\
\end{align}
Therefore,
\begin{align}
E[X_t]&= x - E[D_t] + Q \sum_{n \ge 1} P(T_n < t).
\label{inventorylevel1}
\end{align}
\subsection{Expected Total Cost}
\quad \\
\noindent In order to determine an optimal inventory control strategy, it is necessary to evaluate the associated costs explicitly. 
While Noba and Yamazaki \cite{noba2025stochastic} describe the total cost using a general functional form, here we decompose the total cost according to individual cost components.
\begin{dfn}[\textbf{Expected Inventory Cost}]
Let $C_o,\ C_h,\ C_{s.o.}>0$ denote the unit ordering cost, holding cost, and stockout cost, respectively. 
The total cost incurred up to time $t$ is defined by
\begin{equation}
C_{\mathrm{total}}(a,Q,t)
= \underbrace{Q C_o R_t}_{\text{ordering cost}}
  +\underbrace{ C_h \int_0^t X_s ds }_{\text{holding cost}}
  +\underbrace{C_{s.o.}\int_0^t (-X_s )^+ ds}_{\text{stockout cost}}.
\end{equation}
The expectation $E[C_{\mathrm{total}}(a,Q,t)]$ is called the expected total cost.
\end{dfn}
Assuming that the stockout cost is negligible, the expected total cost becomes
\begin{align*}
E[C_{\mathrm{total}}(a,Q,t)] 
&\simeq C_o Q E[R_t] 
+ C_h E\left[\int_0^t(x - D_s + Q R_s) ds\right]   \\
&= C_o Q E[R_t] + C_h x t 
- C_h E \left[\int_0^t D_s ds \right]
+ C_h Q E \left[ \int_0^t R_s ds \right] \\
&= C_o Q E[R_t] + C_h x t 
- C_h \int_0^t E [D_s] ds 
+ C_h Q E \left[ \int_0^t R_s ds \right].
\end{align*}
By Fubini's theorem,
\begin{equation}
E[R_t]
=E\!\left[\sum_{n\ge1}\mathbf 1_{\{T_n<t\}}\right]
=\sum_{n\ge1}P(T_n<t),
\label{eq:ERt_sum}
\end{equation}
and again by Fubini's theorem,
\begin{equation}
E\!\left[\int_0^t R_s\,ds\right]
=\int_0^t E[R_s]\,ds
=\int_0^t \sum_{n\ge1}P(T_n<s)\,ds.
\label{eq:intER}
\end{equation}
Hence,
\begin{align}
E[C_{\mathrm{total}}(a,Q,t)] 
&= C_o Q \sum_{n\ge1}P(T_n<t) 
+ C_h x t 
- C_h \int_0^t E [D_s]ds \notag\\
&\quad + C_h Q \int_0^t \sum_{n\ge1}P(T_n<s)\,ds.
\label{totalcost1}
\end{align}
\begin{rem}
For any $n\ge1$, the identity
\[
\int_0^t \mathbf 1_{\{T_n<s\}}\,ds
=(t-T_n)^+
=t\,\mathbf 1_{\{T_n<t\}}
- T_n\mathbf 1_{\{T_n<t\}}
\]
holds. 
Taking expectations and summing over $n$, and using Fubini's theorem again due to non-negativity, we obtain
\begin{align}
E\!\left[\int_0^t R_s\,ds\right]
&=\sum_{n\ge1}E\!\left[\int_0^t \mathbf 1_{\{T_n<s\}}\,ds\right]
=\sum_{n\ge1}E[(t-T_n)^+] \notag\\
&=\sum_{n\ge1}\Bigl(t\,P(T_n<t)-E[T_n\mathbf 1_{\{T_n<t\}}]\Bigr).
\label{eq:intR_identity}
\end{align}
Therefore, combining \eqref{eq:intR_identity} and \eqref{eq:intER}, we obtain
\begin{equation}
\sum_{n\ge1}\Bigl(t\,P(T_n<t)-E[T_n\mathbf 1_{\{T_n<t\}}]\Bigr)
=\int_0^t \sum_{n\ge1}P(T_n<s)\,ds.
\end{equation}
\end{rem}

\noindent Regarding the Laplace transform of stopping times, the following well-known result holds.

\begin{theono}[cf. J.Bertoin (1996)~\cite{bertoin1996levy}, Sato (1999)~\cite{sato1999levy}]
Let $\{X_t\}$ be a subordinator (a one-dimensional, possibly killed, non-decreasing Lévy process), and define 
\[
\tau_b^+ = \inf\{t>0 \mid X_t > b\}.
\]
Let $\Psi(\theta)= \log E[\exp(\theta X_1)]$ denote the Laplace exponent. Then,
\begin{equation*}
E[\exp(-s \tau_b^+)] = \exp(-b \Phi(s)),
\quad \text{where } \Psi(\Phi(s))= s , \quad (s>0).
\end{equation*}
\end{theono}
\noindent Using this theorem, we will now examine the Laplace transforms of specific subordinators in detail.

\section{Main Results}
\noindent In \eqref{inventorylevel1} and \eqref{totalcost1}, 
the terms involving $T_n$ are difficult to compute unless its distribution is explicitly known. 
In general, there are only few subordinators for which the distribution of the FPT can be written explicitly. 
In this section, we derive the Laplace transform of the FPT for a specific subordinator.

\subsection{Case where the Cumulative Demand is a Drifted Poisson Process}

\begin{dfn}[\textbf{Cumulative Demand}]
Let $\mu > 0$, $\lambda > 0$, and $\alpha > 0$, and let $\{N_t\}_{t \ge 0}$ be a $\lambda$-Poisson process. 
We define the cumulative demand process $\{D_t\}$ as follows:
\begin{equation}
D_t = \mu t + \alpha N_t.
\end{equation}
\end{dfn}

\noindent We review the special function used in deriving the Laplace transform of the first passage time of $\{D_t\}$.

\begin{dfn}[\textbf{Lambert W function}~\cite{MR1414285}]
The complex function $W(z)$ satisfying
\[
W(z)e^{W(z)} = z
\]
is called the Lambert W function.
\end{dfn}

\noindent By choosing an appropriate branch, $W(z)$ can be regarded as a function from $\mathbb{R}$ to $\mathbb{R}$, and results concerning its derivative are also known.

\begin{lem}[\textbf{Derivative of the Lambert W function}\cite{MR1414285}]
For the real-valued Lambert W function $W(x)$,
\begin{equation*}
\frac{d}{dx}W(x) 
= \frac{1}{(1+W(x))\exp(W(x))}
=\frac{W(x)}{x(1+W(x))} ,
\quad \text{if }x \not=0.
\end{equation*}
\end{lem}
\noindent The Laplace transform of the FPT for a drifted Poisson process can be expressed using the Lambert W function.
\begin{theo} \label{maintheo1}
Let $T_n$ be the FPT of the drifted Poisson process $D_t=\mu t + \alpha N_t$ defined above. 
Let $W(x)$ denote the Lambert W function. 
Then the Laplace transform of $T_n$ is given as follows:
\begin{equation}
E[e^{-sT_n}] 
= \exp
\left\{
-(a+(n-1)Q)
\left(
-\frac{1}{\alpha} W \left( \frac{\alpha \lambda }{\mu}\exp\left(\frac{\alpha}{\mu} \left(s + \lambda\right)\right) \right)+ \frac{1}{\mu}\left(s + \lambda\right)
\right)
\right\}.
\end{equation}
\end{theo}

\begin{proof}
We compute the Laplace exponent.
\begin{equation*}
\psi(\theta) 
= \log E[e^{\theta D_1}]
= \theta \mu  +\lambda  (e^{\theta \alpha}-1)
\end{equation*}

\noindent The Laplace transform of $T_n$ is given by
\begin{equation*}
E[e^{-sT_n}] 
= \exp(-(a+(n-1)Q)\Phi(s)) 
\quad \text{, where} \ \psi(\Phi(s)) = s.
\end{equation*}

\noindent We compute the inverse function of $\psi$.
\begin{align*}
s &= \psi(\theta) = \theta \mu  +\lambda  (e^{\theta \alpha}-1) \\
s &= \theta \mu  +\lambda  (e^{\theta \alpha}-1) \\
\lambda e^{\theta \alpha } &=-\theta \mu + s + \lambda \\
\lambda e^{\theta \alpha } &=-\theta \mu + A \quad \quad (A=s + \lambda) \\
y &= \lambda \exp\left(- \alpha\frac{y-A}{\mu} \right) \quad \quad (y = -\theta \mu + A) \\
y &= \lambda \exp\left(\frac{\alpha A}{\mu} \right)  \exp\left(- \frac{\alpha y}{\mu} \right)  \\
y \exp\left(\frac{\alpha}{\mu} y\right) &= \lambda \exp\left(\frac{\alpha}{\mu} A\right) \\
\frac{\alpha}{\mu} y \exp\left(\frac{\alpha}{\mu} y\right) &= \frac{\alpha \lambda }{\mu}\exp\left(\frac{\alpha}{\mu} A\right) \\
\frac{\alpha}{\mu} y  &= W \left( \frac{\alpha \lambda }{\mu}\exp\left(\frac{\alpha}{\mu} A\right) \right) \quad (W : \text{Lambert } W \text{ function})\\
-\alpha \theta + \frac{\alpha}{\mu}A &= W \left( \frac{\alpha \lambda }{\mu}\exp\left(\frac{\alpha}{\mu} A\right) \right) \\
\theta &=  -\frac{1}{\alpha} W \left( \frac{\alpha \lambda }{\mu}\exp\left(\frac{\alpha}{\mu} A\right) \right)+ \frac{1}{\mu}A
\end{align*}

\noindent Therefore,
\begin{equation*}
\Phi(s)= -\frac{1}{\alpha} W \left( \frac{\alpha \lambda }{\mu}\exp\left(\frac{\alpha}{\mu} \left(s + \lambda\right)\right) \right)+ \frac{1}{\mu}\left(s + \lambda\right).
\end{equation*}

\noindent Hence, the Laplace transform of $T_n$ is given by
\begin{equation*}
E[e^{-sT_n}] 
= \exp
\left\{
-(a+(n-1)Q)
\left(
-\frac{1}{\alpha} W \left( \frac{\alpha \lambda }{\mu}\exp\left(\frac{\alpha}{\mu} \left(s + \lambda\right)\right) \right)+ \frac{1}{\mu}\left(s + \lambda\right)
\right)
\right\}.
\end{equation*}
\end{proof}

\begin{cor} \label{cor1}Under the assumptions of Theorem ~\ref{maintheo1} , it holds that
\begin{align}
E[T_n]
&= \frac{a+(n-1)Q}{\mu + \alpha \lambda},
\end{align}
\begin{align}
\mathrm{Var}[T_n] &= \frac{\alpha^2 \lambda (a+(n-1)Q)}{\left( \mu + \alpha \lambda \right)^3} .
\end{align}
\end{cor}

\begin{proof}
Let
\begin{equation*}
\Phi(s)\coloneqq \frac{1}{\alpha} W \left( \frac{\alpha \lambda }{\mu}\exp\left(\frac{\alpha}{\mu} \left(s + \lambda\right)\right) \right)- \frac{1}{\mu}\left(s + \lambda\right)
,\qquad 
K \coloneqq a + (n-1)Q
\end{equation*}
be defined. We will organize the differentiation rules for the Laplace transform. We have
\begin{align*}
\frac{d}{ds} E[e^{-sT_n}] &= \frac{d}{ds} \exp(-K\Phi(s)) 
= -K \exp(-K\Phi(s)) \frac{d \Phi}{ds} (s) \\
&= -K E[e^{-sT_n}] \frac{d \Phi}{ds} (s) = -K E[e^{-sT_n}] \frac{d \Phi}{ds} (s) = -K  \exp(-K\Phi(s))  \frac{d \Phi}{ds} (s).
\end{align*}
We will also organize the second-order derivatives of the Laplace transform. We get
\begin{align*}
\frac{d^2}{ds^2} E[e^{-sT_n}] &= \frac{d}{ds} \left( -K E[e^{-sT_n}] \frac{d \Phi}{ds} (s) \right) \\ 
&= -K \left(\frac{d}{ds} E[e^{-sT_n}] \frac{d \Phi}{ds} (s) +  E[e^{-sT_n}] \frac{d^2 \Phi}{ds^2} (s) \right) \\
&= -K \left(-K E[e^{-sT_n}] \frac{d \Phi}{ds} (s) \times  \frac{d \Phi}{ds} (s) +  E[e^{-sT_n}] \frac{d^2 \Phi}{ds^2} (s) \right) \\
&= -K E[e^{-sT_n}] \left(-K \left( \frac{d \Phi}{ds} (s) \right)^2+   \frac{d^2 \Phi}{ds^2} (s) \right) \\
&= -K \exp(-K\Phi(s))  \left(-K \left( \frac{d \Phi}{ds} (s) \right)^2+   \frac{d^2 \Phi}{ds^2} (s) \right) .
\end{align*}
\begin{align*}
\frac{d \Phi}{ds} (s) &= \frac{d}{ds} \left( -\frac{1}{\alpha} W \left( \frac{\alpha \lambda }{\mu}\exp\left(\frac{\alpha}{\mu} \left(s + \lambda\right)\right) \right)
+ \frac{1}{\mu}\left(s + \lambda\right) \right) \\
&= -\frac{1}{\alpha} \frac{dW}{ds}   \left( \frac{\alpha \lambda }{\mu}\exp\left(\frac{\alpha}{\mu} \left(s + \lambda\right)\right) \right) 
+ \frac{1}{\mu } .
\end{align*}
Compute the first and second derivatives of $\Phi$. We have
Let $B = \frac{\alpha}{\mu}$ and $C = \frac{\alpha \lambda }{\mu}$.
\begin{align*}
\frac{d \Phi}{ds} (s) 
&=  -\frac{1}{\alpha}\frac{dW}{ds}   \left( C \exp\left(B s +C\right) \right) 
 + \frac{1}{\mu } .
\end{align*}
\noindent Furthermore, 
\begin{align*}
\frac{d^2 \Phi}{ds^2} (s) 
&=  -\frac{1}{\alpha}\frac{d^2W}{ds^2} \left( C \exp\left(B s +C\right) \right).
\end{align*}
\noindent The derivative of $W(C \exp(Bs +C))$ can be computed as follows:
\begin{align*}
\frac{dW}{ds}   \left( C \exp\left(B s +C\right) \right) 
&= \frac{W \left( C \exp\left(B s +C\right) \right)}{1+ W \left( C \exp\left(B s +C\right) \right)}  
\left( C \exp\left(B s +C\right) \right)^{-1} 
\left( BC \exp\left(B s +C\right) \right) \\
&= B\frac{W \left( C \exp\left(B s +C\right) \right)}{1+ W \left( C \exp\left(B s +C\right) \right)}  .
\end{align*}
\noindent Furthermore,
\begin{align*}
&\quad  \frac{d^2W}{ds^2} \left( C \exp\left(B s +C\right) \right) \\
&= B\frac{d}{ds} \left( \frac{W \left( C \exp\left(B s +C\right) \right)}{1+ W \left( C \exp\left(B s +C\right) \right)} \right)\\
&= B \frac{dW}{ds} \left( C \exp\left(B s +C\right) \right) 
\frac{1}{1+ W \left( C \exp\left(B s +C\right) \right)} \\
&\quad \quad \quad \quad \quad \quad \quad \quad 
+ B \times W \left( C \exp \left(B s +C\right) \right) 
\frac{d}{ds} \left( \frac{1}{1+ W \left( C \exp\left(B s +C\right) \right)} \right) \\
&= B \times B\frac{W \left( C \exp\left(B s +C\right) \right)}{1+ W \left( C \exp\left(B s +C\right) \right)}  
\times \frac{1}{1+ W \left( C \exp\left(B s +C\right) \right)} \\
&\quad \quad \quad \quad \quad \quad \quad \quad 
+ B \times W \left( C \exp \left(B s +C\right) \right) 
\times (-1) \frac{ \frac{dW}{ds} \left( C \exp\left(B s +C\right) \right) }{(1+ W \left( C \exp\left(B s +C\right) \right))^2}  \\
&= B \times B\frac{W \left( C \exp\left(B s +C\right) \right)}{1+ W \left( C \exp\left(B s +C\right) \right)}  
\times \frac{1}{1+ W \left( C \exp\left(B s +C\right) \right)} \\
&\quad \quad 
- B \times W \left( C \exp \left(B s +C\right) \right) 
\times  \frac{1}{(1+ W \left( C \exp\left(B s +C\right) \right))^2} \times B\frac{W \left( C \exp\left(B s +C\right) \right)}{1+ W \left( C \exp\left(B s +C\right) \right)}  \\
&= B^2 \frac{W \left( C \exp\left(B s +C\right) \right)}{(1+ W \left( C \exp\left(B s +C\right) \right))^2}   
- B^2 \frac{(W \left( C \exp\left(B s +C\right) \right))^2}{(1+ W \left( C \exp\left(B s +C\right) \right))^3}  \\
&= B^2 \frac{W \left( C \exp\left(B s +C\right) \right)}{(1+ W \left( C \exp\left(B s +C\right) \right))^2}   
\left(1 - \frac{(W \left( C \exp\left(B s +C\right) \right))}{(1+ W \left( C \exp\left(B s +C\right) \right))} \right) \\
&= B^2 \frac{W \left( C \exp\left(B s +C\right) \right)}{(1+ W \left( C \exp\left(B s +C\right) \right))^2}   
\times \frac{1}{1+ W \left( C \exp\left(B s +C\right) \right)}  \\
&= B^2 \frac{W \left( C \exp\left(B s +C\right) \right)}{(1+ W \left( C \exp\left(B s +C\right) \right))^3} .
\end{align*}
\noindent It follows that
\begin{align*}
&\quad \frac{d}{ds} E[e^{-sT_n}] \\
&= -K  \exp(-K\Phi(s))  \frac{d \Phi}{ds} (s) \\
&= -(a+(n-1)Q)  \exp\left[ -(a+(n-1)Q) \left(-\frac{1}{\alpha} W \left( \frac{\alpha \lambda }{\mu}\exp\left(\frac{\alpha}{\mu} \left(s + \lambda\right)\right) \right)+ \frac{1}{\mu}\left(s + \lambda\right) \right)\right]  \\
&\quad \times  
\left(
-\frac{1}{\alpha}B\frac{W \left( C \exp\left(B s +C\right) \right)}{1+ W \left( C \exp\left(B s +C\right) \right)}   + \frac{1}{\mu } 
\right)\\
&= -(a+(n-1)Q)  \exp\left[ -(a+(n-1)Q) \left(-\frac{1}{\alpha} W \left( \frac{\alpha \lambda }{\mu}\exp\left(\frac{\alpha}{\mu} \left(s + \lambda\right)\right) \right)+ \frac{1}{\mu}\left(s + \lambda\right) \right)\right]  \\
&\quad \times  
\left(
-\frac{1}{\mu}
\frac{W \left( \frac{\alpha \lambda }{\mu} \exp\left(\frac{\alpha}{\mu} s +\frac{\alpha \lambda }{\mu}\right) \right)
}
{
1+ W \left( \frac{\alpha \lambda }{\mu} \exp\left(\frac{\alpha}{\mu} s +\frac{\alpha \lambda }{\mu}\right) \right)
}  
+ \frac{1}{\mu } 
\right).
\end{align*}

\noindent Therefore, the expected value of $E[T_n]$ is given by:
\begin{align*}
E[T_n]
&= - \left. \frac{d}{ds} E[e^{-sT_n}] \right| _{s=0}\\
&= (a+(n-1)Q)  \exp\left[ -(a+(n-1)Q) \left(\frac{1}{\alpha} W \left( \frac{\alpha \lambda }{\mu} e^{\frac{\alpha \lambda}{\mu} } \right)- \frac{\lambda}{\mu}\right)\right]  \\
&\quad \times  
\left( -\frac{1}{\mu} \frac{W \left(  \frac{\alpha \lambda }{\mu} e^{\frac{\alpha \lambda }{\mu} } \right)}{1+ W \left(  \frac{\alpha \lambda }{\mu} e^{\frac{\alpha \lambda}{\mu} } \right)} 
+ \frac{1}{\mu } 
\right)\\
&= (a+(n-1)Q)  \exp\left[ -(a+(n-1)Q) \left(\frac{1}{\alpha} \frac{\alpha \lambda }{\mu} - \frac{\lambda}{\mu}\right)\right] \times  
\left( -\frac{1}{\mu} \frac{\frac{\alpha \lambda }{\mu}}{1+ \frac{\alpha \lambda }{\mu}} 
+ \frac{1}{\mu } 
\right)\\
&=\frac{a+(n-1)Q }{\mu + \alpha \lambda}.
\end{align*}
Similarly, calculating the second moment yields:
\begin{align*}
&\quad \frac{d^2}{ds^2} E[e^{-sT_n}] \\
&= \frac{d}{ds} \left( -K E[e^{-sT_n}] \frac{d \Phi}{ds} (s) \right) \\ 
&= -K \exp(-K\Phi(s))  \left(-K \left( \frac{d \Phi}{ds} (s) \right)^2+   \frac{d^2 \Phi}{ds^2} (s) \right) \\
&= - (a+(n-1)Q)
\exp\left[ -(a+(n-1)Q) 
\left(
-\frac{1}{\alpha} W \left( \frac{\alpha \lambda }{\mu} e^{\frac{\alpha}{\mu} \left(s + \lambda\right)} \right)
+ \frac{1}{\mu}\left(s + \lambda\right) 
\right)\right] \\
&\quad \times\left(
-(a+(n-1)Q) \left( 
-\frac{1}{\alpha} \frac{dW}{ds}   \left( C \exp\left(B s +C\right) \right) + \frac{1}{\mu} 
\right)^2
\right.\\
&\quad \left.
-\frac{1}{\alpha} B^2 \frac{W \left( C \exp\left(B s +C\right) \right)}{(1+ W \left( C \exp\left(B s +C\right) \right))^3}   
\right) \\
&= - (a+(n-1)Q)
\exp\left[ -(a+(n-1)Q) 
\left(
-\frac{1}{\alpha} W \left( \frac{\alpha \lambda }{\mu} e^{\frac{\alpha}{\mu} \left(s + \lambda\right)} \right)
+ \frac{1}{\mu}\left(s + \lambda\right) 
\right)\right] \\
&\quad \times\left(
-(a+(n-1)Q) \left( 
-\frac{1}{\mu}
\frac{W \left( \frac{\alpha \lambda }{\mu} \exp\left(\frac{\alpha}{\mu} s +\frac{\alpha \lambda }{\mu}\right) \right)
}
{
1+ W \left( \frac{\alpha \lambda }{\mu} \exp\left(\frac{\alpha}{\mu} s +\frac{\alpha \lambda }{\mu}\right) \right)
}  
+ \frac{1}{\mu } \right)^2 \right. \\
&\quad
\left.
-\frac{\alpha}{\mu^2} \frac{W \left( \frac{\alpha \lambda }{\mu} \exp\left(\frac{\alpha}{\mu} s +\frac{\alpha \lambda }{\mu}\right) \right)}
{\left(1+ W \left( \frac{\alpha \lambda }{\mu} \exp\left(\frac{\alpha}{\mu} s +\frac{\alpha \lambda }{\mu}\right) \right)\right)^3}   
\right).
\end{align*}
Hense,
\begin{align*}
E[T_n^2]
&= - (a+(n-1)Q)
\times\left(
-(a+(n-1)Q) \left( 
-\frac{1}{\mu}
\frac{W \left(  \frac{\alpha \lambda }{\mu} e^{\frac{\alpha \lambda }{\mu} } \right)}{1+ W \left(  \frac{\alpha \lambda }{\mu} e^{\frac{\alpha \lambda}{\mu} } \right)} 
+ \frac{1}{\mu } \right)^2 \right. \\
&\quad
\left.
-\frac{\alpha}{\mu^2} 
\frac{W \left(  \frac{\alpha \lambda }{\mu} e^{\frac{\alpha \lambda }{\mu} } \right)}{\left(1+ W \left(  \frac{\alpha \lambda }{\mu} e^{\frac{\alpha \lambda}{\mu} } \right)\right)^3} 
\right) \\
&= - (a+(n-1)Q)
\times\left(
-(a+(n-1)Q) \left( 
-\frac{1}{\mu}
\frac{\frac{\alpha \lambda }{\mu}}{1+ \frac{\alpha \lambda }{\mu}} 
+ \frac{1}{\mu } \right)^2 \right. \\
&\quad
\left.
-\frac{\alpha}{\mu^2} 
\frac{\frac{\alpha \lambda }{\mu}}{\left(1+ \frac{\alpha \lambda }{\mu}\right)^3} 
\right) \\
&= \left( \frac{a+(n-1)Q}{\mu + \alpha \lambda } \right)^2
+\frac{\alpha^2 \lambda (a+(n-1)Q)}{\left( \mu + \alpha \lambda \right)^3}.
\end{align*}
From these results, we obtain
\noindent 
\begin{align*}
\Var[T_n] 
= E[T_n^2] - E[T_n]^2 
= \frac{\alpha^2 \lambda (a+(n-1)Q)}{\left( \mu + \alpha \lambda \right)^3}.
\end{align*}

\end{proof}

\subsection{Case where the Cumulative Demand is a Drifted Compound Poisson Process}
\quad \\
\noindent As a generalization of the drifted Poisson process, considering a drifted compound Poisson process is not only a mathematical generalization but also a valuable extension in practice, since it can reproduce situations in which the magnitude of sudden demand changes is random each time.

\begin{dfn}[Cumulative Demand]
Let $\mu > 0$, $\lambda > 0$, and let $\{N_t\}_{t \ge 0}$ be a $\lambda$-Poisson process.  
Let $\{J_k\}_{n\ge 1}$ be a sequence of random variables that are independent of $\Dt$ and independent and identically distributed.  
We define the cumulative demand process $\Dt$ as follows:
\begin{equation}
D_t = \mu t + \sum_{k=1}^{N_t} J_k. \label{driftedcompound}
\end{equation}
\end{dfn}

\noindent In this case, let $T_n$ be the FPT defined above, and let $\psi(\theta)$ be the Laplace exponent of $D_t$. Then,
\begin{equation*}
E[e^{-sT_n}] = \exp(-(a+(n-1)Q)\Phi(s)) \quad \text{,where} \ \psi(\Phi(s)) = s.
\end{equation*}

\noindent Let $L_{J_1}(\theta)$ denote the Laplace transform of $J_1$, and compute the Laplace exponent $\psi(\theta)$.
\begin{equation*}
\psi(\theta) = \log E[e^{\theta D_1}]
= \log E \left[\exp\left( \theta \mu + \theta \sum_{k=1}^{N_1} J_k \right)\right] 
= \theta \mu +\log E \left[\exp\left(\theta \sum_{k=1}^{N_1} J_k \right)\right]
\end{equation*}
\begin{align*}
E \left[\exp\left(\theta \sum_{k=1}^{N_1} J_k \right)\right]
&= E\left[ E \left[\exp\left(\theta \sum_{k=1}^{N_1} J_k \right) \ \middle| \ N_1 \right] \right] \\
&= \sum_{m \ge 0} E \left[\exp\left(\theta \sum_{k=1}^{m} J_k \right)\right]P(N_1 = m) \\
&= \sum_{m \ge 0} E \left[\exp(\theta J_1)\right]^m P(N_1 = m) \quad \because \ J_k \ \text{i.i.d.}\\
&= \sum_{m \ge 0} {L_{J_1}(\theta)}^m e^{-\lambda} \frac{\lambda^m}{m !} \\
&= e^{-\lambda} \exp(\lambda L_{J_1}(\theta))\\
\end{align*}
\noindent Therefore,
\begin{align*}
\log E \left[\exp\left(\theta \sum_{k=1}^{N_1} J_k \right)\right] 
&= \log \exp (\lambda(L_{J_1}(\theta)-1)) \\
&= \lambda(L_{J_1}(\theta)-1).
\end{align*}
\noindent Hence,
\begin{equation}
\psi(\theta) = \theta \mu +\lambda(L_{J_1}(\theta)-1) .\label{laplace_exponent}
\end{equation}

\subsection*{Case where the Jump Sizes Follow an Exponential Distribution}

\begin{theo} \label{maintheo2}
Let $\Dt$ be defined as in (\ref{driftedcompound}), and assume that $J_k$ are independent and identically distributed and follow an exponential distribution with parameter $\eta$.  
In this case, for $\theta < \eta$,
\begin{equation}
E[e^{-sT_n}] = \exp\left(-(a+(n-1)Q)\frac{1}{2 \mu } 
\left\{ 
(\mu \eta + s + \lambda)- \sqrt{(\mu \eta + s + \lambda)^2 - 4 \mu s \eta}
\right\}\right)  .
\end{equation} 
\end{theo}

\begin{proof}
Since $L_{J_1}(\theta)=\frac{\eta}{\eta-\theta}$, it follows from (\ref{laplace_exponent}) that
\begin{equation*}
\psi(\theta) = \theta \mu +\lambda\left(\frac{\eta}{\eta-\theta}-1\right) = \theta\mu + \frac{\lambda \theta}{\eta - \theta} .
\end{equation*}

\noindent We compute the inverse function of $\psi$.
\begin{align*}
s = \psi(\theta) &= \theta \mu +\lambda\left(\frac{\eta}{\eta-\theta}-1\right) = \theta\mu + \frac{\lambda \theta}{\eta - \theta} \\
\theta \mu (\eta- \theta)+ \lambda \theta &= s (\eta - \theta) \\
\mu \theta^2 - (\mu \eta + s + \lambda)\theta + s \eta &= 0 \\
\Phi(s) &= \theta 
= \frac{1}{2 \mu } 
\left\{ 
(\mu \eta + s + \lambda)- \sqrt{(\mu \eta + s + \lambda)^2 - 4 \mu s \eta}
\right\} .
\end{align*}

\noindent The other solution of the quadratic equation does not satisfy $\theta < \eta$, and is therefore not appropriate. Hence,
\begin{align*}
E[e^{-sT_n}] &= \exp(-(a+(n-1)Q)\Phi(s)) \\
&= \exp\left(-(a+(n-1)Q)
\frac{1}{2 \mu } 
\left\{ 
(\mu \eta + s + \lambda)- \sqrt{(\mu \eta + s + \lambda)^2 - 4 \mu s \eta}
\right\}\right).
\end{align*}
\end{proof}

\begin{cor} Under the assumptions of Theorem ~\ref{maintheo2}, it holds that
\begin{equation}
E[T_n] =  \frac{\eta(a+(n-1)Q)}{\mu \eta  + \lambda} ,
\end{equation}
\begin{equation}
\Var[T_n]= \frac{2\eta \lambda (a+(n-1)Q)}{(\mu \eta  + \lambda)^3} .
\end{equation}
\end{cor}

\begin{proof}The calculation follows the same approach as Corollary~\ref{cor1}.
Let
\begin{equation*}
\Phi(s) \coloneqq
\frac{1}{2 \mu } 
\left\{ 
(\mu \eta + s + \lambda)- \sqrt{(\mu \eta + s + \lambda)^2 - 4 \mu s \eta}
\right\},
\quad K \coloneqq a + (n-1)Q
\end{equation*}
Recall the following result regarding the differentiation of the Laplace transform of $T_n$:
\begin{align*}
\frac{d}{ds} E[e^{-sT_n}] 
&= -K  \exp(-K\Phi(s))  \frac{d \Phi}{ds} (s).
\end{align*}
\begin{align*}
\frac{d^2}{ds^2} E[e^{-sT_n}]
= -K \exp(-K\Phi(s))  \left(-K \left( \frac{d \Phi}{ds} (s) \right)^2+   \frac{d^2 \Phi}{ds^2} (s) \right) .
\end{align*}
\noindent Compute the first and second derivatives of $\Phi$.
\begin{align*}
\frac{d \Phi}{ds} (s) 
&= \frac{1}{2 \mu } 
\left\{ 
 1 - \frac{2(\mu \eta + s + \lambda) - 4 \mu \eta}{2 \sqrt{(\mu \eta + s + \lambda)^2 - 4 \mu s \eta} }
\right\} \\
&= \frac{1}{2 \mu } 
\left\{ 
 1 - \frac{2s + 2\lambda - 2 \mu \eta}{2 \sqrt{(\mu \eta + s + \lambda)^2 - 4 \mu s \eta} }
\right\} \\
&= \frac{1}{2 \mu } 
\left\{ 
 1 - \frac{s + \lambda - \mu \eta}{\sqrt{(\mu \eta + s + \lambda)^2 - 4 \mu s \eta} }
\right\}.
\end{align*}

\begin{align*}
\frac{d^2 \Phi}{ds^2}(s) 
&= - \frac{1}{2 \mu } 
\times
\frac{\left\{
\sqrt{(\mu \eta + s + \lambda)^2 - 4 \mu s \eta} - (s+ \lambda - \mu \eta )\frac{(s+ \lambda - \mu \eta )}{\sqrt{(\mu \eta + s + \lambda)^2 - 4 \mu s \eta} }
\right\}}
{\left\{(\mu \eta + s + \lambda)^2 - 4 \mu s \eta \right\}^{\frac{3}{2}} }\\
&= - \frac{1}{2 \mu } 
\times
\frac{\left\{
(\mu \eta + s + \lambda)^2 - 4 \mu s \eta - (s+ \lambda - \mu \eta )^2
\right\}}
{\left\{(\mu \eta + s + \lambda)^2 - 4 \mu s \eta \right\}^{\frac{3}{2}} }\\
&= - \frac{1}{2 \mu } 
\times
\frac{\left\{
4\mu \eta( s + \lambda) - 4 \mu s \eta 
\right\}}
{\left\{(\mu \eta + s + \lambda)^2 - 4 \mu s \eta \right\}^{\frac{3}{2}} }\\
&= - \frac{
2\eta( s + \lambda) - 2 \mu s \eta 
}
{\left\{(\mu \eta + s + \lambda)^2 - 4 \mu s \eta \right\}^{\frac{3}{2}} }.
\end{align*}

\noindent Moreover,We verify the value at 0.
\begin{equation*}
\Phi(0) = 0 .
\end{equation*}

\begin{align*}
\frac{d \Phi}{ds} (0)  &= \frac{1}{2 \mu } 
\left\{ 
 1 - \frac{ \lambda - \mu \eta}{\mu \eta  + \lambda} 
\right\} \\
&= \frac{1}{2 \mu } 
\left\{ 
 \frac{ 2\mu \eta}{\mu \eta  + \lambda} 
\right\} = \frac{\eta}{\mu \eta  + \lambda}.
\end{align*}

\begin{align*}
\frac{d^2 \Phi}{ds^2}(0) 
= - \frac{2\eta \lambda }{\left\{(\mu \eta  + \lambda)^2 \right\}^{\frac{3}{2}} }
= - \frac{2\eta \lambda }{(\mu \eta  + \lambda)^3 }.
\end{align*}

\noindent Therefore, the expectation is given by
\begin{align*}
E[T_n]
= - \left. \frac{d}{ds} E[e^{-sT_n}] \right| _{s=0}
= K  \exp(-K\Phi(0))  \frac{d \Phi}{ds} (0)
=  \frac{\eta(a+(n-1)Q)}{\mu \eta  + \lambda}.
\end{align*}
Regarding the second moment, we have
\begin{align*}
E[T_n ^2]
&=\left.\frac{d^2}{ds^2} E[e^{-sT_n}] \right|_{s=0} \\
&= -K \exp(-K\Phi(0))  \left(-K \left( \frac{d \Phi}{ds} (0) \right)^2+   \frac{d^2 \Phi}{ds^2} (0) \right) \\
&= (a+(n-1)Q)^2 \left( \frac{\eta}{\mu \eta  + \lambda} \right)^2
+(a+(n-1)Q)\frac{2\eta \lambda }{(\mu \eta  + \lambda)^3} \\
&=  \frac{\eta^2(a+(n-1)Q)^2}{(\mu \eta  + \lambda)^2 } 
+\frac{2\eta \lambda (a+(n-1)Q)}{(\mu \eta  + \lambda)^3} .
\end{align*}
As for the variance, we obtain
\begin{align*}
\Var[T_n] 
= E[T_n^2] - E[T_n]^2 
= \frac{2\eta \lambda (a+(n-1)Q)}{(\mu \eta  + \lambda)^3}.
\end{align*}
\end{proof}

\subsection*{Case where the Jump Sizes Follow a Gamma Distribution}
\quad \\
\noindent As a generalization of Theorem~\ref{maintheo2},
we consider a compound Poisson process whose jump sizes follow a Gamma distribution.
In this case, the Laplace transform cannot be expressed in closed form.
However, since its derivatives at the origin can be computed explicitly,
the expectation and variance can still be obtained.
\begin{theo} \label{maintheo3}
Let $\Dt$ be defined as in (\ref{driftedcompound}), and assume that $J_k$ are independent and identically distributed and follow a Gamma distribution with parameters $\beta$ and $\eta$.  
Then,
\begin{align*}
E[T_n] &= \frac{\eta(a+(n-1)Q)}{\mu \eta + \beta \lambda }, \\
\mathrm{Var}[T_n] &= \frac{\beta(\beta+1)\eta \lambda (a+(n-1)Q)}{(\mu \eta  + \beta \lambda)^3}.
\end{align*}
\end{theo}

\begin{proof}
\noindent Let $L_{J_1}= \left(\frac{\eta}{\eta - \theta}\right)^\beta$. Then, for $\theta< \eta$, the Laplace Exponent $\Psi$ can be calculated as follows:
\begin{equation*}
\psi(\theta) = \theta \mu +\lambda\left(\left(\frac{\eta}{\eta - \theta}\right)^\beta -1\right).
\end{equation*}

\begin{equation*}
\frac{d\psi}{d\theta}(\theta) =  \mu + \lambda \beta \eta^{\beta}(\eta - \theta)^{-(\beta+1)}.
\end{equation*}

\begin{equation*}
\frac{d^2\psi}{d\theta^2}(\theta) = \lambda \beta(\beta +1) \eta^{\beta}(\eta - \theta)^{-(\beta+2)}.
\end{equation*}

\noindent Let $\psi^{-1} = \Phi$ denote the inverse function of $\psi$. Then, it holds that
\begin{equation*}
\Phi(0) = 0 \quad (\because \psi(0)=0),
\end{equation*}
\begin{equation*}
\frac{d\Phi}{ds}(0) = \frac{1}{\cfrac{d\psi}{d\theta}(0)} 
=\frac{1}{\mu + \lambda \beta \eta^{-1}}
=\frac{\eta}{\mu \eta  +  \beta \lambda },
\end{equation*}
\begin{equation*}
\frac{d^2 \Phi}{ds^2}(0) 
= -\frac{\cfrac{d^2\psi}{d\theta^2}(0)}{\left(\cfrac{d\psi}{d\theta}(0)\right)^3}
= -\frac{\lambda \beta(\beta + 1)\eta^{-2}}{(\mu + \lambda \beta \eta^{-1})^3} 
= -\frac{\beta(\beta + 1)\eta \lambda }{(\mu \eta + \beta \lambda )^3}.
\end{equation*}

\noindent Therefore, the expectation is given by
\begin{align*}
E[T_n] &= - \left. \frac{d}{ds} E[e^{-sT_n}] \right|_{s=0}=  K  \exp(-K\Phi(0))  =\frac{d \Phi}{ds} (0) \frac{\eta (a+(n-1)Q)}{\mu \eta  +  \beta \lambda }.
\end{align*}
The second moment is given by
\begin{align*}
E[T_n^2] &= \left. \frac{d^2}{ds^2} E[e^{-sT_n}] \right|_{s=0} \\
&= -K \exp(-K\Phi(0))  \left(-K \left( \frac{d \Phi}{ds} (0) \right)^2+   \frac{d^2 \Phi}{ds^2} (0) \right) \\
&= -(a+(n-1)Q)  \left(-(a+(n-1)Q) \left( \frac{\eta}{\mu \eta  +  \beta \lambda } \right)^2   -\frac{\beta(\beta + 1)\eta \lambda }{(\mu \eta + \beta \lambda )^3}\right) \\
&=  \left( \frac{\eta(a+(n-1)Q) }{\mu \eta  +  \beta \lambda } \right)^2+ 
\frac{\beta(\beta + 1)\eta \lambda(a+(n-1)Q) }{(\mu \eta + \beta \lambda )^3}.
\end{align*}
The variance is given by
\begin{align*}
\Var[T_n] = E[T_n^2] - E[T_n]^2 
= \frac{\beta(\beta + 1)\eta \lambda(a+(n-1)Q) }{(\mu \eta + \beta \lambda )^3}.
\end{align*}
\end{proof}

\subsection{Generalization of the Cumulative Demand}
\quad \\
\noindent Based on the previous discussion, we further generalize the model.

\begin{dfn}[\textbf{Generalized Cumulative Demand}]
Let $\mu , \alpha , \lambda ,\lambda'  > 0$.  
Let $\Nt$ be a $\lambda$-Poisson process,  
$\{N'_t\}_{t \ge 0 }$ be a $\lambda'$-Poisson process,  
and $\{J_k\}_{k \ge 1}$ be a sequence of independent and identically distributed random variables.  
Assume that $\Nt$, $\{N'_t\}_{t \ge 0 }$, and $\{J_k\}_{k \ge 1}$ are mutually independent.  
We define the cumulative demand process $\Dt$ as
\begin{equation}
D_t = \mu t + \alpha N_t + \sum_{k=1}^{N'_t} J_k. 
\label{generaldemand}
\end{equation}
\end{dfn}
\noindent In this case, for 
\[
T_n = \inf \left\{s > T_{n-1} \,\middle|\, D_s \ge a + (n-1)Q \right\},
\]
we derive $E[e^{-sT_n}]$, $E[T_n]$, and $\Var[T_n]$.
Let $\psi(\theta)$ denote the Laplace exponent of $D_t$. Then,
\begin{equation*}
E[e^{-sT_n}] = \exp(-(a+(n-1)Q)\Phi(s)) 
\quad \text{,where} \ \psi(\Phi(s)) = s.
\end{equation*}

\noindent Let $L_{J_1}(\theta)$ denote the Laplace transform of $J_1$.  
We compute the Laplace exponent $\psi(\theta)$:
\begin{equation*}
\psi(\theta) = \log E[e^{\theta D_1}]
= \theta \mu  +\lambda  (e^{\theta \alpha}-1)  +\lambda'(L_{J_1}(\theta)-1).
\end{equation*}

\begin{theo} \label{maintheo4}
Let $\Dt$ be defined as in (\ref{generaldemand}), and assume that $J_k$ are independent and identically distributed and follow a Gamma distribution with parameters $\beta,\eta$. Then,
\begin{equation}
E[T_n]
= \frac{\eta(a+(n-1)Q)}{\mu \eta +\alpha \lambda + \beta \lambda' },
\end{equation}
\begin{equation}
\Var[T_n] 
= \frac{(\alpha^2 \eta^3 \lambda  + \beta(\beta+1)\eta \lambda') (a+(n-1)Q)}{(\mu \eta  +\alpha \eta \lambda+ \beta \lambda')^3}.
\end{equation}
\end{theo}

\begin{proof}
\noindent Since $L_{J_1}= \left(\frac{\eta}{\eta - \theta}\right)^\beta$, for $\theta< \eta$, the Laplace Exponent $\Psi$ can be calculated as follows:
\begin{equation*}
\psi(\theta) 
= \log E[e^{\theta D_1}]
= \theta \mu  +\lambda  (e^{\theta \alpha}-1)  
+\lambda'\left(\left(\frac{\eta}{\eta - \theta}\right)^\beta-1\right),
\end{equation*}

\begin{equation*}
\frac{d\psi}{d\theta}(\theta) 
=  \mu + \alpha \lambda e^{\alpha \theta}+ \lambda' \beta \eta^{\beta}(\eta - \theta)^{-(\beta+1)},
\end{equation*}

\begin{equation*}
\frac{d^2\psi}{d\theta^2}(\theta) 
=  \alpha^2 \lambda e^{\alpha \theta}+ \lambda' \beta(\beta +1) \eta^{\beta}(\eta - \theta)^{-(\beta+2)}.
\end{equation*}

\noindent Let $\Phi$ denote the inverse function of $\psi$. Then, it holds that
\begin{equation*}
\Phi(0) = 0 \quad (\because \psi(0)=0),
\end{equation*}

\begin{equation*}
\frac{d\Phi}{ds}(0) 
= \frac{1}{\cfrac{d\psi}{d\theta}(0)} 
=\frac{1}{\mu + \alpha \lambda+ \lambda' \beta \eta^{-1}}
=\frac{\eta}{\mu \eta  + \alpha \lambda \eta+  \beta \lambda' },
\end{equation*}

\begin{equation*}
\frac{d^2 \Phi}{ds^2}(0) 
= -\frac{\cfrac{d^2\psi}{d\theta^2}(0)}{\left(\cfrac{d\psi}{d\theta}(0)\right)^3}
= -\frac{\alpha^2\lambda + \lambda' \beta(\beta + 1)\eta^{-2}}
{(\mu + \alpha\lambda +\lambda' \beta \eta^{-1})^3} 
= -\frac{\alpha^2\eta^3 \lambda +\beta(\beta + 1)\eta \lambda' }
{(\mu \eta  + \alpha\lambda\eta+ \beta \lambda')^3}.
\end{equation*}

\noindent Therefore, we have
\begin{align*}
E[T_n] 
&= - \left. \frac{d}{ds} E[e^{-sT_n}] \right|_{s=0}=  K  \exp(-K\Phi(0))  \frac{d \Phi}{ds} (0) = \frac{\eta(a+(n-1)Q)}{\mu \eta +\alpha \lambda + \beta \lambda' } ,
\end{align*}

\begin{align*}
E[T_n^2] 
&= \left. \frac{d^2}{ds^2} E[e^{-sT_n}] \right|_{s=0} \\
&= -K \exp(-K\Phi(0))  
\left(-K \left( \frac{d \Phi}{ds} (0) \right)^2+   \frac{d^2 \Phi}{ds^2} (0) \right) \\
&= -(a+(n-1)Q)  
\left(-(a+(n-1)Q) 
\left( \frac{\eta}{\mu \eta +\alpha \lambda + \beta \lambda' } \right)^2   \right. \\
&\quad -\left.
\frac{(\alpha^2 \eta^3 \lambda  + \beta(\beta+1)\eta \lambda') }
{(\mu \eta  +\alpha \eta \lambda+ \beta \lambda')^3}
\right) \\
&=  \left( \frac{\eta(a+(n-1)Q)}
{\mu \eta +\alpha \lambda + \beta \lambda' } \right)^2
+ \frac{(\alpha^2 \eta^3 \lambda  + \beta(\beta+1)\eta \lambda') 
(a+(n-1)Q)}
{(\mu \eta  +\alpha \eta \lambda+ \beta \lambda')^3}.
\end{align*}
For the variance, we obtain
\begin{align*}
\Var[T_n] 
= E[T_n^2] - E[T_n]^2 
= 
\frac{(\alpha^2 \eta^3 \lambda  + \beta(\beta+1)\eta \lambda') 
(a+(n-1)Q)}
{(\mu \eta  +\alpha \eta \lambda+ \beta \lambda')^3}.
\end{align*}

\end{proof}

\subsection{On the Expected Total Cost}
\quad \\
\noindent Although we do not directly use the distribution of $T_n$, we consider the generalized cumulative demand process and investigate the expected value of the total inventory cost derived from it.
Consider the cumulative demand process
\[
D_t=\mu t+\alpha N_t+\sum_{k=1}^{N'_t}J_k.
\]
Here, $\{N_t\}_{t\ge0}$ is a $\lambda$-Poisson process,
$\{N'_t\}_{t\ge0}$ is a $\lambda'$-Poisson process,
and $\{J_k\}_{k\ge1}$ is a sequence of nonnegative independent and identically distributed random variables.
We assume that these processes are mutually independent.

\begin{lem}
For any $s>0$ and $b>0$,
\begin{equation}
P(T_n<s)=P(D_s\ge b)
=
\sum_{i=0}^{\infty}\sum_{j=0}^{\infty}
p_i(s)\,q_j(s)\,
P\!\left(
\sum_{k=1}^{j}J_k \ge b-\mu s-\alpha i
\right)
\label{distTn}
\end{equation}
holds, where
\[
p_i(s)=e^{-\lambda s}\frac{(\lambda s)^i}{i!},
\qquad
q_j(s)=e^{-\lambda' s}\frac{(\lambda' s)^j}{j!}.
\]
\end{lem}

\begin{proof}
\noindent Since the cumulative demand process $D_t$ is nondecreasing,
$\{T_n<s\}\iff\{D_s\ge b\}$ holds,
and thus $P(T_n<s)=P(D_s\ge b)$.

\[
P(D_s\ge b)
=
\sum_{i=0}^{\infty}\sum_{j=0}^{\infty}
P(D_s\ge b \mid N_s=i,N'_s=j)\,
P(N_s=i,N'_s=j).
\]

\noindent Because the Poisson processes $N_s$ and $N'_s$ are independent,
\[
P(N_s=i,N'_s=j)
=
P(N_s=i)P(N'_s=j)
=
p_i(s)\,q_j(s).
\]

\noindent Under the condition $(N_s,N'_s)=(i,j)$,
\[
D_s=\mu s+\alpha i+\sum_{k=1}^{j}J_k
\quad \text{a.s.}
\]
\noindent Therefore,
\[
P(D_s\ge b \mid N_s=i,N'_s=j)
=
P\!\left(
\sum_{k=1}^{j}J_k \ge b-\mu s-\alpha i
\right).
\]
\end{proof}

\begin{lem}
Let $\{J_k\}_{k\ge1}$ be independent and identically distributed random variables following a Gamma distribution with parameters $\beta,\eta>0$.
Then for any integer $j\ge1$ and $x\in \mathbb{R}$,
\begin{equation}
P\!\left(\sum_{k=1}^{j}J_k \ge x\right)
=
\begin{cases}
1, & x\le0,\\[6pt]
\dfrac{\Gamma(j\beta,\eta x)}{\Gamma(j\beta)}, & x>0,
\end{cases}
\end{equation}
where $\Gamma(\cdot,\cdot)$ denotes the upper incomplete gamma function.

\noindent In particular, when $\beta =1$,
\begin{equation}
P\!\left(\sum_{k=1}^{j}J_k \ge x\right)
=
e^{-\eta x}\sum_{r=0}^{j-1}\frac{(\eta x)^r}{r!}.
\end{equation}
\end{lem}

\begin{proof}
If $x\le0$, since $J_k\ge0$, it is clear that
$\sum_{k=1}^jJ_k\ge x$.

\noindent Assume $x>0$.
Using the Laplace transform of the Gamma distribution
\[
E[e^{sJ_1}]
=\left(\frac{\eta}{\eta-s}\right)^{\beta}
\]
and independence, we obtain
\[
E\!\left[e^{-s\sum_{k=1}^jJ_k}\right]
=\left(\frac{\eta}{\eta-s}\right)^{j\beta}.
\]
\noindent This coincides with the Laplace transform of the Gamma distribution with parameters $j\beta ,\eta$.
Therefore,
\[
\sum_{k=1}^{j}J_k
\]
follows a Gamma distribution with parameters $j\beta ,\eta$.
Hence,
\[
P\!\left(\sum_{k=1}^{j}J_k \ge x\right)
=
\int_x^\infty
\frac{\eta^{j\beta}}{\Gamma(j\beta)}
u^{j\beta-1}e^{-\eta u}\,du.
\]
\noindent By the change of variables $y=\eta u$,
\[
P\!\left(\sum_{k=1}^{j}J_k \ge x\right)
=
\frac{1}{\Gamma(j\beta)}
\int_{\eta x}^\infty y^{j\beta-1}e^{-y}\,dy
=
\frac{\Gamma(j\beta,\eta x)}{\Gamma(j\beta)}.
\]

\noindent When $\beta=1$, $j\beta=j$ is an integer.
Consider
\[
\Gamma(j,z)=\int_z^\infty y^{j-1}e^{-y}\,dy.
\]
Let $I_{j-1}(z):=\int_z^\infty y^{j-1}e^{-y}\,dy$.

\[
I_{j-1}(z)
=
\Bigl[-y^{j-1}e^{-y}\Bigr]_z^\infty
+(j-1)\int_z^\infty y^{j-2}e^{-y}\,dy.
\]

\noindent Since $\lim_{y\to\infty}y^{j-1}e^{-y}=0$,
\[
I_{j-1}(z)
=
z^{j-1}e^{-z}
+(j-1)I_{j-2}(z).
\]

\noindent Repeating integration by parts recursively yields
\[
I_{j-1}(z)
=
e^{-z}\Bigl(
z^{j-1}
+(j-1)z^{j-2}
+(j-1)(j-2)z^{j-3}
+\cdots
+(j-1)!
\Bigr).
\]

\noindent Rearranging coefficients gives
\[
I_{j-1}(z)
=
(j-1)!e^{-z}
\sum_{r=0}^{j-1}\frac{z^r}{r!}.
\]

\noindent Thus,
\[
\Gamma(j,z)=(j-1)!e^{-z}\sum_{r=0}^{j-1}\frac{z^r}{r!}.
\]

\noindent Therefore,
\[
\Gamma(j,\eta x)
=
(j-1)!e^{-\eta x}\sum_{r=0}^{j-1}\frac{(\eta x)^r}{r!}.
\]
\end{proof}

\begin{rem}
The same result can also be obtained by performing a direct computation in the case where the jump sizes $J_k$ follow an exponential distribution.
\end{rem}

\begin{theo}
Consider the cumulative demand process
\[
D_t=\mu t+\alpha N_t+\sum_{k=1}^{N'_t}J_k.
\]
Here, $\{N_t\}_{t\ge0}$ is a $\lambda$-Poisson process,
$\{N'_t\}_{t\ge0}$ is a $\lambda'$-Poisson process,
and $\{J_k\}_{k\ge1}$ is a sequence of nonnegative independent and identically distributed random variables.
Assume that these are mutually independent. 
Then, for any $s>0$ and $b>0$,
\begin{align}
E[C_{\mathrm{total}}(a,Q,t)]
&=
C_o Q
\sum_{n\ge1}P(D_t\ge a+(n-1)Q) \notag\\
&\quad
+ C_h x t
- \frac{C_h t^2}{2}
\bigl(\mu+\alpha\lambda+\lambda'E[J_1]\bigr) \notag\\
&\quad
+ C_h Q
\int_0^t
\sum_{n\ge1}P(D_s\ge a+(n-1)Q)\,ds.
\end{align}
\end{theo}

\begin{proof}
The result follows from (\ref{totalcost1}) and (\ref{distTn}).
\end{proof}

\noindent If the distribution of $J_k$ is specified explicitly, we obtain the following corollary.

\begin{cor}
Let $\{J_k\}_{k\ge1}$ be independent and identically distributed random variables following a Gamma distribution with parameters $\beta,\eta>0$.
Then the expected total cost is given by
\begin{align}
&\quad E[C_{\mathrm{total}}(a,Q,t)] \notag\\
&=
C_o Q
\sum_{n\ge1}
\sum_{i=0}^{\infty}\sum_{j=0}^{\infty}
p_i(t)q_j(t)\,
P\!\left(
\sum_{k=1}^{j}J_k
\ge
a+(n-1)Q-\mu t-\alpha i
\right)\notag\\
&\quad
+ C_h x t
-
\frac{C_h t^2}{2}
\Bigl(\mu+\alpha\lambda+\lambda'\tfrac{\beta}{\eta}\Bigr) \notag\\
&\quad
+
C_h Q
\int_0^t
\sum_{n\ge1}
\sum_{i=0}^{\infty}\sum_{j=0}^{\infty}
p_i(s)q_j(s)\,
P\!\left(
\sum_{k=1}^{j}J_k
\ge
a+(n-1)Q-\mu s-\alpha i
\right)
ds,
\end{align}
where
\[
p_i(s)=e^{-\lambda s}\frac{(\lambda s)^i}{i!},
\qquad
q_j(s)=e^{-\lambda' s}\frac{(\lambda' s)^j}{j!},
\]
and
\[
P\!\left(
\sum_{k=1}^{j}J_k \ge y
\right)
=
\begin{cases}
1, & y\le0,\\[6pt]
\dfrac{\Gamma(j\beta,\eta y)}{\Gamma(j\beta)}, & y>0,
\end{cases}
\]
holds.In particular, when $\beta =1$,
\begin{align}
E[C_{\mathrm{total}}(a,Q,t)]
&=
C_o Q
\sum_{n\ge1}
\sum_{i=0}^{\infty}\sum_{j=0}^{\infty}
p_i(t)q_j(t)\,
\Phi_j(a+(n-1)Q-\mu t-\alpha i) \notag \\
&\quad
+ C_h x t
-
\frac{C_h t^2}{2}
\Bigl(\mu+\alpha\lambda+\tfrac{\lambda'}{\eta}\Bigr) \notag \\
&\quad
+
C_h Q
\int_0^t
\sum_{n\ge1}
\sum_{i=0}^{\infty}\sum_{j=0}^{\infty}
p_i(s)q_j(s)\,
\Phi_j(a+(n-1)Q-\mu s-\alpha i)
ds.
\end{align}

\noindent Here,
\[
\phi_j(y)
=
P\!\left(\sum_{k=1}^{j}J_k \ge y\right)
=
\begin{cases}
1, & y \le 0,\\[6pt]
\displaystyle
e^{-\eta y}
\sum_{r=0}^{j-1}\frac{(\eta y)^r}{r!},
& y>0.
\end{cases}
\]
\end{cor}

\begin{theo}[Comparison between Finite-Horizon and Long-Run Average]
Assume $E[J_1]<\infty$ and define
\[
m \coloneq \bigl(\mu+\alpha\lambda+\lambda' E[J_1]\bigr).
\]
Then,
\begin{equation}
\lim_{t\to\infty}\frac{1}{t} E[C_{\mathrm{total}}(a,Q,t)]
=
C_o m + C_h x,
\end{equation}
and this limit does not depend on $Q$ (nor on $a$).
\end{theo}

\begin{proof}
By monotonicity,
\[
\{T_n<t\}\iff \{D_t\ge a+(n-1)Q\}
\]
holds, and therefore
\[
 E[R_t]
=\sum_{n\ge1} P(T_n<t)
=\sum_{n\ge1} P(D_t\ge a+(n-1)Q).
\]

\noindent Hence, from the previously derived representation of the total cost,
\[
 E[C_{\mathrm{total}}(a,Q,t)]
=
C_o Q\, E[R_t]
+ C_h x t
- C_h\int_0^t E[D_s]\,ds
+ C_h Q\int_0^t  E[R_s]\,ds.
\]

\noindent Next, we compute $ E[D_s]$. By independence and linearity,
\[
 E[D_s]
=\mu s+\alpha\, E[N_s]+ E[N'_s] E[J_1]
=\bigl(\mu+\alpha\lambda+\lambda' E[J_1]\bigr)s =: ms.
\]
\noindent Thus,
\[
\int_0^t E[D_s]\,ds=\frac{m}{2}t^2.
\]

\noindent To establish the long-run average, we provide upper and lower bounds relating $R_t$ and $D_t$.
Since $R_t$ counts the number of thresholds reached by $D_t$, we have
\[
a+(R_t-1)Q \le D_t < a+R_t Q.
\]
Therefore,
\[
\frac{D_t-a}{Q} \le R_t \le \frac{D_t-a}{Q}+1.
\]

\noindent Taking expectations and dividing by $t$,
\[
\frac{1}{t}\frac{ E[D_t]-a}{Q}
\le
\frac{ E[R_t]}{t}
\le
\frac{1}{t}\frac{ E[D_t]-a}{Q}+\frac{1}{t}.
\]

\noindent Since $ E[D_t]=mt$, letting $t\to\infty$ yields
\[
\lim_{t\to\infty}\frac{ E[R_t]}{t}=\frac{m}{Q}.
\]

\noindent Furthermore,
\[
\frac{1}{t^2}\int_0^t E[R_s]\,ds
=\int_0^1 \frac{ E[R_{ut}]}{t}\,du
\longrightarrow
\int_0^1 \frac{m}{Q}u\,du
=\frac{m}{2Q}.
\]
\noindent Hence,
\[
\int_0^t E[R_s]\,ds \sim \frac{m}{2Q}t^2.
\]
\noindent Substituting these into the total cost expression,
\begin{align*}
\frac{1}{t} E[C_{\mathrm{total}}(a,Q,t)]
&=
C_o Q\frac{ E[R_t]}{t}
+ C_h x
-\frac{C_h}{2}mt
+ C_h Q\cdot \frac{1}{t}\int_0^t E[R_s]\,ds.
\end{align*}

\noindent We have
\[
C_o Q\frac{ E[R_t]}{t}\to C_o m,
\]
and
\[
C_h Q\cdot \frac{1}{t}\int_0^t E[R_s]\,ds
=
C_h Q\cdot t\cdot \frac{1}{t^2}\int_0^t E[R_s]\,ds
\to
C_h Q\cdot t\cdot \frac{m}{2Q}
=\frac{C_h}{2}mt.
\]

\noindent Thus, the terms proportional to $t$ cancel, and we obtain
\[
\lim_{t\to\infty}\frac{1}{t} E[C_{\mathrm{total}}(a,Q,t)]
= C_o m + C_h x
= C_o \bigl(\mu+\alpha\lambda+\lambda' E[J_1]\bigr) + C_h x.
\]

\noindent Since the limit does not contain $Q$, the long-run average cost does not depend on $Q$ (nor on $a$).
\end{proof}

\begin{cor}
If $J_k$ follows an exponential distribution with parameter $\eta$,
\begin{equation}
\lim_{t\to\infty}\frac{1}{t} E[C_{\mathrm{total}}(a,Q,t)]
= C_o \left(\mu+\alpha\lambda+\frac{\lambda'}{\eta} \right) + C_h x.
\end{equation}

\noindent If $J_k$ follows a Gamma distribution with parameters $\beta,\eta$,
\begin{equation}
\lim_{t\to\infty}\frac{1}{t} E[C_{\mathrm{total}}(a,Q,t)]
= C_o \left(\mu+\alpha\lambda+\frac{\beta\lambda'}{\eta} \right) + C_h x.
\end{equation}
\end{cor}

\section*{Acknowledgements}
\noindent The author is deeply grateful to
Professor Seiji Hiraba and Takumu Ooi
for their continuous support, insightful suggestions,
and invaluable guidance.
The author also thanks Professor Aya Ishigaki and Yuriko Ono
for fruitful collaboration and stimulating discussions.

\nocite{*}
\bibliographystyle{plain} 
\bibliography{ref}  

\end{document}